\documentclass[notitlepage,leqno,12pt]{article}

\usepackage{latexsym}
\usepackage{amsmath}
\usepackage{amsfonts}
\usepackage{amssymb}
\usepackage{amsthm}
\usepackage{color}

\renewcommand{\a}{\alpha}
\renewcommand{\b}{\beta}
\renewcommand{\d}{\delta}
\newcommand{\D}{\Delta}

\newcommand{\e}{\varepsilon}
\newcommand{\g}{\gamma}

\newcommand{\Ric} {{\rm Ric}}

\newcommand{\n}{\nabla}
\newcommand{\var}{\varphi}

\newcommand{\s}{\sigma}

\newcommand{\ov}{\overline}

\newcommand{\be}{\begin{equation}}
\newcommand{\ee}{\end{equation}}

\newcommand{\R}{\mathbb{R}}

\newcommand{\Z}{\mathbb{Z}}

\renewcommand{\P}{\mathbb{P}}

\newcommand{\N}{\mathbb{N}}

\newcommand{\la}{\langle}
\newcommand{\ra}{\rangle}
\newcommand{\Riem}{\mathrm{Riem}}
\newcommand{\Sc}{\mathrm{Sc}}
\newcommand{\Cat}{\mathrm{Cat}}

\newcommand{\rd}{\mathrm{d}}

\newcommand{\nab}{\nabla}

\newcommand{\cO}{\mathfrak{O}}

\DeclareMathOperator{\Hess}{\mathrm{Hess}\,}

\newtheorem{thm}{Theorem}[section]

\newtheorem{lmm}[thm]{Lemma}

\newtheorem{prp}[thm]{Proposition}

\theoremstyle{definition}

\theoremstyle{remark}
\newtheorem{ex}[thm]{Example}
\theoremstyle{remark}
\newtheorem{rem}[thm]{Remark}
\theoremstyle{remark}

\usepackage[margin=1.9cm, bottom=3cm, footskip=1.5cm]{geometry}
\usepackage{layout}

\author{Norihisa Ikoma 
\\ 
\footnotesize Keio University 
\\
\footnotesize \texttt{ikoma@math.keio.ac.jp}
\and  Andrea Malchiodi \\ \footnotesize Scuola Normale Superiore, Pisa\\
\footnotesize \texttt{andrea.malchiodi@sns.it} 
\and  Andrea Mondino \\ \footnotesize University of Warwick  \\ \footnotesize \texttt{A.Mondino@warwick.ac.uk}}

\title{Foliation by area-constrained Willmore spheres near a non-degenerate critical point of the scalar curvature}

\begin{document}

\hyphenation{un-coun-ta-bly ma-ni-folds ma-ni-fold se-cond eigen-values po-si-ti-ve}

\maketitle

\begin{abstract}
Let $(M,g)$ be a 3-dimensional Riemannian manifold. The goal of the paper it to show that if $P_{0}\in M$ is a non-degenerate critical point of the scalar curvature, then a neighborhood of $P_{0}$ is foliated by area-constrained Willmore spheres. 
Such a foliation is \emph{unique} among foliations by area-constrained Willmore spheres having Willmore energy less than $32\pi$,  moreover it is \emph{regular} in the sense that a suitable rescaling smoothly converges to a round sphere in the Euclidean three-dimensional space. We also establish generic multiplicity of foliations and the first multiplicity result for area-constrained Willmore spheres with prescribed (small) area in a closed Riemannian manifold. The topic has strict links with the Hawking mass. 
\end{abstract}

\begin{center}

\bigskip


\bigskip

\centerline{\bf AMS subject classification: } 
49Q10, 53C21, 53C42, 35J60, 83C99.
\end{center}

\section{Introduction}

Let $\Sigma$ be a closed (compact, without boundary) two-dimensional surface, 
$(M,g)$ a 3-dimensional Riemannian manifold and $f: \Sigma \to M$ a smooth immersion. 
The  \emph{Willmore functional} $W(f)$  is defined by 
\be\label{eq:defW}
W(f):=\int_{\Sigma} H^2 \, d\sigma.
\ee
Here $d \sigma$ is the area form induced by $f$, 
$H:= \bar{g}^{ij} A_{ij}$ the mean curvature, 
$\bar{g}_{ij}$ the induced metric and $A_{ij}$ the second fundamental form. 
The immersion $f$ is said to be a \emph{Willmore surface} (or \emph{Willmore immersion}) if it is a critical point of $W$ with respect to normal variations or, equivalently, when it satisfies 
the Euler-Lagrange equation 
\be\label{eq:WillmoreEq}
\Delta_{\bar{g}} H + H |\mathring{A}|^2 + H \Ric(n, n)=0,
\ee
where $\Delta_{\bar{g}}$ is the Laplace-Beltrami operator, 
$\mathring{A}_{ij} := A_{ij} - \frac{1}{2} H \bar{g}_{ij}$ the trace free 
second fundamental form, $n$ a unit normal to $f$ and 
$\Ric$ the Ricci tensor of $(M,g)$. The Willmore equation \eqref{eq:WillmoreEq} is a fourth-order nonlinear elliptic PDE 
in the immersion map $f$. 
\\

The Willmore energy \eqref{eq:defW}  appears not only in mathematics but also 
in various fields of science and technology. 
For example, in biology, it is  
 a special case of the \emph{Helfrich energy} (\cite{H,KMR,S}). 
In general relativity, the \emph{Hawking mass} contains the Willmore functional 
as the main term (see below for the definition of the Hawking mass) and  in String Theory the Polyakov's extrinsic action involves the functional as well.
\\

The Willmore functional was first introduced in the XIXth century in the Euclidean ambient space  by Sophie Germain in her work on elasticity.  Blaschke and Thomsen, in the 1920s-30s,  detected the class of Willmore surfaces as a natural conformally invariant generalization of minimal surfaces. Indeed minimal surfaces are solutions of \eqref{eq:WillmoreEq} 
(as $H \equiv 0$), and the Willmore functional $W$ in the Euclidean space 
is conformally invariant (provided the center of the inversion does not lie on the surface, in which case one has to add a constant depending on the multiplicity of the immersion).

After that, Willmore rediscovered the topic in 1960s. 
He proved that round spheres are the only global minimizers of $W$ 
among all closed immersed surfaces into the Euclidean space (see \cite{W-93}) and he conjectured that 
the Clifford torus and its images under the M\"obius transformations 
are the global minimizers among surfaces of higher genus. 
The Willmore conjecture was recently solved by Marques-Neves \cite{MN} 
through minimax techniques. 
Let us also mention other fundamental works on the Willmore functional in the Euclidean ambient space. 
Simon \cite{SiL} proved the existence of a smooth genus-one minimizer of  $W$ in $\R^m$ and developed a general regularity theory for minimizers. 
The existence of a minimizer for every genus was settled by  Bauer-Kuwert \cite{BK}, Kusner \cite{Kus} and  Rivi\`ere \cite{Riv1,Riv2} who also developed an independent regularity theory holding more generally for stationary points of $W$. 
We also wish to mention the work  by Kuwert-Sch\"atzle \cite{KS} 
on the Willmore flow and by  Bernard-Rivi\`ere \cite{BR} and Laurain-Rivi\`ere \cite{LauRiv}
on bubbling and energy-identities phenomena. 
\\

Let us emphasize that all the aforementioned results concern Willmore  immersions into the \emph{Euclidean space}, or equivalently into a round sphere 
due to the conformal invariance. The literature about Willmore immersions into curved Riemannian manifolds, which has interest in
applications as it might model non-homogeneous environments, is much more recent.
The first existence result in ambient space with non-constant sectional curvature was \cite{M-10}, where  
the third author showed the existence of embedded Willmore spheres 
(Willmore surface with genus equal to zero) 
in a perturbative setting. 
We also refer to \cite{M-JGA} and \cite{CM} in collaboration with Carlotto 
for related results. 
Under the area constraint condition, the existence of Willmore type spheres 
and their properties have been investigated by Lamm-Metzger \cite{LM-10, LM-13}, Lamm-Metzger-Schulze \cite{LMS-11}, and
the third author in collaboration with Laurain \cite{LauMon}. The existence of area-constrained Willmore tori of small size have been recently addressed by the authors of this work in \cite{IMM-1,IMM-2}.

The global problem, i.e. the existence of smooth immersed spheres minimizing quadratic curvature functionals in compact 3-dimensional Riemannian manifolds,  was  studied by the third author in collaboration with Kuwert and Schygulla in \cite{KMS} (see also \cite{MonSch} for the non compact case).  In  collaboration with Rivi\`ere \cite{MR1,MR2}, the third author  developed the necessary tools for the calculus of variations of the Willmore functional in Riemannian manifolds and proved the existence of area-constrained Willmore spheres in homotopy classes  (as well as the existence of Willmore spheres under various assumptions and  constraints).
\\

The present paper, as well as the aforementioned works
\cite{LM-10,LM-13, LMS-11,LauMon,MR1,MR2}, concerns the existence of Willmore spheres 
under area constraint.  Such immersions satisfy the equation
\begin{equation*}\label{eq:ConstrWill}
\Delta_{\bar{g}} H + H |\mathring{A}|^2 + H \Ric(n, n)=\lambda H,
\end{equation*}
for some $\lambda\in \R$ playing the role of Lagrange multiplier. These immersions are strictly related  to the Hawking mass
\begin{equation*}\label{eq:defHawMass}
m_H(f):=\frac{\sqrt{Area(f)}}{64 \pi^{3/2}} \left(16\pi-W(f)\right),
\end{equation*}
in the sense that critical points of the Hawking mass under the area constraint condition 
are equivalent to the area-constrained Willmore immersions. 
We refer to \cite{ChYau,LMS-11} and the references therein 
for more material about the latter topic.
\\

In order to motivate our main theorems, let us discuss more in detail the literature which is closest to our new results.

 \begin{itemize}
\item  Lamm-Metzger \cite{LM-13} proved that, given a closed 3-dimensional Riemannian manofold $(M,g)$, there exists $\varepsilon_0>0$ with the following property: for every $\e\in (0,\varepsilon_0]$ there exists an area-constrained Willmore sphere minimizing the Willmore functional among immersed spheres of area equal to $4 \pi \e^2$.  Moreover, as $\varepsilon \searrow 0$, such area-constrained Willmore spheres concentrate to a maximum point of the scalar curvature and, after suitable rescaling, they converge in $W^{2,2}$-sense to a round sphere.
\item  The above result has been generalized in two ways.  On the one hand Rivi\`ere and the third author in \cite{MR1,MR2} proved that it is possible to minimize the Willmore energy among (bubble trees of possibly branched weak) immersed spheres of fixed area, for every positive value of the area. On the other hand Laurain and the third author in \cite{LauMon} showed that any sequence of area-constrained Willmore spheres with areas converging to zero  and Willmore energy strictly below $32 \pi$ (no matter if they minimize the Willmore energy) have to concentrate  to a critical point of the scalar curvature and, after suitable rescaling, they converge \emph{smoothly} to a round sphere.
\end{itemize}
Some natural questions then arise:
\begin{enumerate}
\item Is it true that around any critical point $P_{0}$ of the scalar curvature one can find a sequence of area-constrained Willmore spheres having area equal to $4 \pi \varepsilon_n^2\to 0$ and concentrating at $P_{0}$?
\item More precisely, can one find a foliation of a neighborhood of  $P_{0}$ made by area-constrained Willmore spheres? 
\item What about  uniqueness/multiplicity?
\end{enumerate}

The goal of the present paper is exactly to investigate the questions 1,2,3 above.  More precisely, on the one hand we reinforce the assumption by asking that $P_{0}$ is a \emph{non-degenerate} critical point of the scalar curvature (in the sense that the Hessian expressed in local coordinates is an invertible matrix);  on the other hand we do not just prove the existence of area-constrained  Willmore spheres concentrating at $P_{0}$ but we show  that there exists a \emph{regular foliation} of a neighborhood of $P_{0}$ made by area-constrained  Willmore spheres. The precise statement is the following.

\begin{thm}\label{thm:foliation}
Let $(M,g)$ be a $3$-dimensional Riemannian manifold and let $P_{0}\in M$ be a non-degenerate critical point of the scalar curvature $\Sc$. Then there exist $\varepsilon_0>0$ and  a neighborhood  $U$ of $P_{0}$ such that $U\setminus \{P_{0}\}$  is foliated by area-constrained Willmore spheres $\Sigma_\varepsilon$ having area  $4\pi \varepsilon^2$, $\varepsilon\in (0,\varepsilon_0)$. More precisely, there is a diffeomorphism $F:S^2\times (0,\varepsilon_0) \to U\setminus \{P_{0}\}$ such that $\Sigma_\varepsilon:=F(S^2,\varepsilon)$ is an area-constrained Willmore sphere  having area equal to   $4\pi \varepsilon^2$. Moreover
\begin{itemize}
\item If the index of $P_{0}$ as a critical point of $\Sc$ is equal to $3-k$ \footnote{The index of a non-degenerate critical point $P_{0}$ of a function $h:M\to \R$ is the number of negative eigenvalues of the Hessian of $h$ at $P_{0}$} , then each surface $\Sigma_\varepsilon$ is an area-constrained critical point of $W$ of  index $k$. 
\item If $\Sc_{P_{0}}>0$ then the surfaces $\Sigma_\varepsilon$ have strictly positive Hawking mass.
\item  The foliation is \emph{regular} at $\varepsilon=0$ in the  following sense. Fix a system of normal coordinates of $U$ centered at $P_{0}$ and indentify $U$ with  an open  subset of $\R^3$;   then, called $F_\varepsilon:=\frac{1}{\varepsilon} F(\cdot, \varepsilon):S^2 \to \R^3$, as $\varepsilon \searrow 0$ the immersions  $F_\varepsilon$ converge smoothly to the round unit sphere of $\R^3$ centered at the origin.
\item The foliation is \emph{unique} in the following sense.  Let  $V \subset U$ be  another neighborhood of $P_{0} \in M$  such that $V\setminus \{P_{0}\}$  is foliated by area-constrained Willmore spheres $\Sigma'_\varepsilon$ having area  $4\pi \varepsilon^2$, $\varepsilon\in (0,\varepsilon_1)$, and satisfying $\sup_{\varepsilon\in (0,\varepsilon_1)} W(\Sigma'_\varepsilon) < 32 \pi$. Then there exists $\varepsilon_2\in (0, \min(\varepsilon_0,\varepsilon_1))$ such that $\Sigma_\varepsilon=\Sigma'_\varepsilon$ for every $\varepsilon \in (0,\varepsilon_2)$.
\end{itemize}
\end{thm}    

Foliations by area-constrained Willmore spheres have been recently investigated by  Lamm-Metzger-Schulze \cite{LMS-11} who proved that a non-compact 3-dimensional manifold which is  asymptotically Schwartzschild with positive mass is foliated at infinity by area-constrained Willmore spheres of large area. Even if both ours and theirs construction rely on a suitable application of the Implicit Function Theorem, the two results and proofs are actually quite different. Theorem \ref{thm:foliation} gives a local foliation in a small neighborhood of a point and the driving geometric quantity is the scalar curvature. 
On the other hand, the main result in \cite{LMS-11}  is a foliation at infinity and the driving geometric quantity is the ADM mass of the manifold.   

 Local foliations by spherical surfaces in manifolds have already appeared in the literature, but mostly by constant mean curvature spheres. In particular we have been inspired by the seminal paper of Ye \cite{Ye}, producing a local foliation of constant mean curvature spheres near a non-degenerate critical point of the scalar curvature. On the other hand let us stress the difference between the two problems: finding a foliation by constant mean curvature spheres is a \emph{second} order problem since the mean curvature is a second order elliptic operator, while finding a foliation by area-constrained Willmore spheres  is a \emph{fourth} order problem.
\\

\noindent
Let us also discuss the relevance of Theorem \ref{thm:foliation} in connection with the Hawking mass. From the note
of Christodoulou and Yau \cite{ChYau}, if $(M, g)$ has non-negative scalar curvature then isoperimetric spheres (and
more generally stable CMC spheres) have positive Hawking mass; it is also known (see for
instance \cite{Druet} or \cite{NardAGA}) that, if $M$ is compact, then small isoperimetric regions converge to geodesic spheres
centered at a maximum point of the scalar curvature as the enclosed volume converges to 0. Moreover, from the aforementioned paper of Ye \cite{Ye} it follows that near a  non-degenerate maximum point of the scalar curvature one can find a foliation by stable CMC spheres, which in particular by \cite{ChYau} will have positive Hawking mass.  Therefore a link between Hawking mass and critical
points of the scalar curvature was already present in literature; Theorem \ref{thm:foliation} expresses this relation
precisely. 
\\

\noindent
We also establish the multiplicity of area-constrained Willmore spheres and generic multiplicity of foliations. Let us mention that, despite the rich literature about existence of area-constrained Willmore spheres,  this is the first multiplicity  result  in general Riemannian manifolds (for a prescribed value of the area constraint).

\begin{thm}\label{thm:multS}
Let $(M,g)$ be a closed $3$-dimensional Riemannian manifold. Let 
\begin{itemize}
\item k=2,  if $M$ is simply connected (i.e. if and only if  $M$ is diffeomorphic to $S^3$ by the recent proof of the Poincar\'e conjecture);
\item k=3, if $\pi_1(M)$ is a non-trivial free group;
\item k=4, otherwise.
\end{itemize}
 Then there exists $\varepsilon_0>0$  such that for every $\varepsilon \in (0,\varepsilon_0)$ there exist at least $k$ distinct area-constrained Willmore spheres of area $4\pi \varepsilon^2$. 
\end{thm}

\begin{rem}
Examples of manifolds having a non-trivial free group as the fundamental group are 
$M = (S^1 \times S^2) \# \cdots \#  (S^1 \times S^2) $ (the connected sum of $m$ copies of $S^1\times S^2$, $m\geq 1$). 
On the other hand, the 3-dimensional real projective space $\R\P^3$ and the 3-torus $S^1\times S^1\times S^1$ are 
instead an example of manifold where $k=4$. 
An expert reader will observe that $k=\Cat(M)+1$, where $\Cat(M)$ is the Lusternik-Schnirelmann category of $M$. 
This is not by chance, indeed Theorem \ref{thm:multS} is proved 
by combining a Lyapunov-Schmidt reduction with the celebrated  Lusternik-Schnirelmann theory.
\end{rem}

\noindent 
We conclude with a remark about  generic multiplicity  of foliations. To this aim note that, fixed a compact manifold $M$, for generic metrics the scalar curvature is a Morse function.

\begin{rem}\label{rem:multFol}
Let $(M,g)$ be a closed 3-dimensional manifold such that the scalar curvature $\Sc:M\to \R$ is a Morse function  and denote with $b_k(M)$ the $k^{th}$ Betti number of $M$, $k=0,\ldots,3$.  Then, by the Morse inequalities,  $\Sc$ has at least $b_k(M)$ non-degenerate critical points of index $k$ and, by Theorem \ref{thm:foliation},  each one of these  points has an associated foliation by area-constrained Willmore spheres of index $3-k$. In particular  there exists $\varepsilon_0>0$ such that, for $\varepsilon \in (0,\varepsilon_0)$, there exist $b_k(M)$ distinct area-constrained Willmore spheres of area $4\pi \varepsilon^2$  and index $3-k$, for $k=0,\ldots,3$; therefore there exist at least $\sum_{k=0}^3 b_k(M)$ distinct area-constrained Willmore spheres of area $4\pi \varepsilon^2$.
\end{rem}

\begin{ex}  Since the Morse inequalities hold by taking the Betti numbers with coefficients in any field, we are free to choose $\R$ or $\Z_2:=\Z/2\Z$ depending on convenience. Let us discuss some basic examples to illustrate the multiplicity statement 
in Remark \ref{rem:multFol}.  
\begin{itemize}
\item $M=S^3$. Then $b_0(M,\R)=b_3(M,\R)=1$, $b_1(M,\R)=b_2(M,\R)=0$ so generically there exits 2 distinct foliations of  area-constrained Willmore spheres.
\item $M=S^2\times S^1$. Then $b_k(M,\R)=1$ for $k=0,\ldots,3$, so generically there exist 4 distinct foliations of  area-constrained Willmore spheres.
\item $M=\R\P^3$. Then $b_k(M,\Z_2)=1$ for $k=0,\ldots,3$, so generically there exist 4 distinct foliations of  area-constrained Willmore spheres.
\item $M=S^1\times S^1\times S^1$. Then $b_k(M,\R)=1$ for $k=0,3$ and $b_k(M,\R)=3$ for $k=1,2$, so generically there exist 8 distinct foliations of area-constrained Willmore spheres.
\end{itemize}
\end{ex}
An announcement of this paper was given  in \cite{IMM-3}. In the independent work \cite{LMS} 
the authors obtained independently some results related to ours.

\subsubsection*{Acknowledgements}
The first author would like to thank Kenta Hayano for discussion on the result of \cite{GG}. 
The authors would like to express their gratitude to Tobias Lamm for letting them know \cite{LMS}. 
The work of the first author was partially supported by JSPS KAKENHI Grant Numbers JP16K17623. 
The third author acknowledges the support of   the EPSRC First Grant  EP/R004730/1.  


\section{Preliminaries}
\label{2000}


	We first recall the definition and properties of the Willmore energy. 
Given a 3-dimensional Riemannian manifold $(M,g)$ and a closed  surface $\Sigma$ immersed in  $M$, 
the Willmore energy $W_g(\Sigma)$ is defined as 
	\[
		W_g(\Sigma) := \int_{\Sigma} H^2 d \s 
	\]
where $H$ is the mean curvature, $H={\rm tr}\,(A)$ where 
$A$ is the second fundamental form of $\Sigma$. 
Here we use the following convention for $A$: 
	\[
		A(X,Y) = g( \nab_X n, Y )
	\]
and $n$ is a  (possibly just locally defined) unit normal to $\Sigma$. 
We also denote by $W_{g_0}$ 
the Willmore energy in the Euclidean space $(\R^3,g_0)$. 
For $W_{g_0}$, we have 
	\begin{prp}[\cite{W-93}]\label{201}
		For any immersed, closed  surface $\Sigma \subset \R^3$, 
		one has 
			\[
				W_{g_0}(\Sigma) \geq 16 \pi = W_{g_0}(S^2)
			\]
		where $S^2 \subset \R^3$ is the standard sphere of unit radius.
	\end{prp}

	Next we recall the first and second variation formulas for $W_g$. 
To be more precise, let $\Sigma \subset M$ be an immersed, closed and 
orientable surface and $F:(-\d,\d) \times \Sigma \to M$ denote 
a perturbation of $\Sigma$ satisfying 
$\partial_t F = \var \, n$ where $n=n(t,p)$ is a unit normal to 
$F(t,\Sigma)$ and $\var := g(n,\partial_t F)$. 
We write $\Riem$ for the Riemann curvature tensor of $M$, 
$\Ric$ the Ricci tensor, 
$\Sc$ the scalar curvature, $\mathring{A}$ 
the traceless second fundamental form, 
$\bar{g}$ the induced metric on $\Sigma$ and 
$\Delta$ the Laplace--Beltrami operator on $(\Sigma,\bar{g})$. 
For $\Riem$, we use the following convention: 
	\[
		\Riem(X,Y)Z = \nab_X \nab_Y Z - \nab_Y \nab_X Z 
		- \nab_{[X,Y]} Z. 
	\]
Moreover, we define a self-adjoint elliptic operator $L$ by 
	\[
		L \varphi := - \Delta \varphi - ( |A|^2 + \Ric(n,n) ) \varphi
	\]
and write $\varpi$ for the tangential component of the one-form 
$\Ric(n,\cdot)$: $\varpi = \Ric(n,\cdot)^t$.
Finally, define the $(2,0)$ tensor $T$ by 
	\[
		T_{ij} = \Riem (\partial_i, n , n , \partial_j) 
		= \Ric_{ij} + G(n,n) \bar{g}_{ij}
	\]
where $G = \Ric - (\frac{1}{2} \Sc) g$ stands for the Einstein tensor of $M$.

	Using these notations, we have the following formulas  
	\begin{prp}[see Section 3 in \cite{LMS-11}]\label{202}
		With the above notation we have
		$$
		  \delta W_g(\Sigma)[\var] := \frac{d}{dt} W_g(F(t,\Sigma)) \Big|_{t=0} 
		  = \int_{\Sigma } \left( L H + \frac{1}{2} H^3 \right) \var \, d \s 
		$$
		and 
		\[
			\begin{aligned}
				 \delta^2 W_g(\Sigma)[\var,\var] 
				 &:= \frac{d^2}{dt^2} W_g(F(t,\Sigma)) \Big|_{t=0} 
				 \\
				 & =  
				 2 \int_{\Sigma } \left[ (L \var)^2 + 
				 \frac{1}{2} H^2 |\n \var |^2 - 2 \mathring{A}(\n \var , \n \var ) \right] 
				 d \s \\ 
		  		& \quad + 2 \int_{\Sigma} \var^2  \bigg( |\n  H|^2_{g} + 2 \varpi(\n H) 
				+ H \D H + 2 g(\n^2 H, \mathring{A}) + 2 H^2 |\mathring{A}|^2_{g} 
				\\
				& \quad +  2 H g (\mathring{A}, T) - H  g (\n_n  \Ric) (n, n) 
				- \frac{1}{2} H^2 |A|^2_{g} - \frac{1}{2} H^2 \Ric(n, n) \bigg) d \s  
				\\
				& \quad + \int_{\Sigma} \left(L H + \frac{1}{2}H^3\right) 
				\left( \frac{\partial \varphi}{\partial t}\bigg|_{t=0} + H \varphi^2 \right) d \s
				\\
		 		& = 2 \int_{\Sigma} \var \tilde{L} \var \, d \s 
		 		+ \int_{\Sigma} \left(L H + \frac{1}{2}H^3\right) 
		 		\left( \frac{\partial \varphi}{\partial t}\bigg|_{t=0} + H \varphi^2 \right) d \s
			\end{aligned}
		\]
		where the fourth-order operator $\tilde{L}$ is defined by 
			\[ 
				\begin{aligned}
					\tilde{L} \var & = L L \var + \frac{1}{2} H^2 L \var 
					+ 2 H g (\mathring{A}, \nabla^2 \var) + 2 H \varpi(\n \var ) 
					+ 2 \mathring{A}(\n \var, \n H) \\
					& \quad +  \var  \Big( |\n H|^2_{g} + 2 \varpi (\n H) + H \D H 
					+ 2 g(\n^2 H, \mathring{A}) \\
					& \qquad \qquad \qquad  
					 + 2 H^2 |\mathring{A}|^2_{g} 
					+ 2 H g(\mathring{A}, T) - H (\n_n \Ric)(n, n) \Big). 
				\end{aligned}
			\]
	\end{prp}

		For  later use, we make some comments 
in the case $(M,g) = (\R^3,g_0)$ and $\Sigma = S^2$. 
In this case, it is easily seen that 
	\[
		\delta W_{g_0}(S^2) = 0, \qquad 
		\delta^2 W_{g_0}(S^2)[\varphi,\varphi] 
		= \int_{S^2} \varphi \tilde{L}_0 \varphi \rd s, 
		\quad \tilde{L}_0 \varphi := \Delta ( \Delta + 2 ) \varphi.
	\]
Furthermore, we see 
	\begin{equation}\label{21}
		 {\rm Ker} \, \tilde{L}_0 = \{ Z_0, Z_1,Z_2, Z_3\} =: \mathcal{K}_0
	\end{equation}
where $\Delta Z_0 = 0 = (\Delta + 2) Z_i$ $(1\leq i \leq 3)$ and 
$Z_j$ ($0\leq j \leq 4$) are given by 
	\begin{equation}\label{22}
		Z_0 (q) \equiv 1 = \frac{H_0(q)}{2}, \quad 
		Z_i (q) = g_0( \mathbf{e}_i, q ) \quad 
		{\rm for}\ q \in S^2, 
	\end{equation}
where $H_0$ is the mean curvature of $S^2$ and 
$\{\mathbf{e}_1,\mathbf{e}_2,\mathbf{e}_3\}$ the canonical basis of $\R^3$. 
These properties will be used in Section \ref{3000}.


\section{Finite-dimensional reduction procedure}
\label{3000}


In this section we perform a Lyapunov--Schmidt reduction  in order
to reduce our problem into a finite-dimensional one, see \cite{AM} for 
a general introduction to this method. 
For most part of this paper we will work on a neighborhood of $P_0 \in M$ 
where $P_0$ is a non-degenerate critical point of $\Sc$, so we start by analyzing this scenario. 
Let us fix such a neighborhood $U_0$ with  
$\overline{U}_0$  compact. 
Next, shrinking $U_0$ if necessary, 
we may find a local orthonormal frame $\{F_{P,1},F_{P,2},F_{P,3}\}_{P \in U_0}$.  By using this frame, we may identify $T_PM$ with $\R^3$ and 
define the exponential map $\exp_P^g: B_{\rho_0}(0) \to M$ 
where $B_{\rho_0}(0)$ is a ball in the Euclidean space and 
$\rho_0>0$ independent of $P \in U_0$. 
We select a neighborhood $V_0$ of $P_0$ and $\e_0>0$ so that 
$\overline{V}_0 \subset U_0$, $\e S^2 \subset B_{\rho_0/2}(0)$ and 
$\exp^{g}_P ( B_{2 \e} (0) ) \subset U_0$ 
for every $0 < \e \leq \e_0$ and $P \in \overline{V}_0$. 
Moreover, since $P_0$ is a non-degenerate critical point of $\Sc$, 
we may assume that 
the Hessian $\Hess (\Sc)$ of $\Sc$ on $\overline{V}_0$ is invertible 
and $P_0$ is the only  critical point of $\Sc$ in $\overline{V}_0$.

		Next, we introduce the following metric $g_\e$ 
which is useful when we observe objects 
satisfying the small area constraint: 
	\[
		g_\e(P) = \frac{1}{\e^2} g(P) \quad {\rm for\ } P \in U_0.
	\]
As  above, $\{\e F_{P,1}, \e F_{P,2} , \e F_{P,3}  \}$ is an orthonormal 
frame for $g_\e$ and we may regard $\exp_{P}^{g_\e}$ as the map 
from some open neighborhood in $\R^3$ into $M$. 
Writing $g_P := (\exp_P^g)^\ast g$ and 
$g_{\e,P} := (\exp_P^{g_\e})^\ast g_\e$ for the pull-backs of $g$ and $g_\e$ 
through the exponential maps, we can check the following: 
(see \cite{IMM-1,LP-87})

	\begin{enumerate}
		\item[(i)] 
		Let $W_{g_\e}$ be the Willmore functional for $(U_0,g_\e)$ and 
		$\Sigma \subset U_0$ be an embedded surface. 
		Denote by $H_g$ and $H_{g_\e}$ the mean curvature of $\Sigma$ 
		in $g$ and $g_\e$, respectively. 
		Then one has 
			\[
				H_{g_\e} = \e H_g, \quad 
				W_{g_\e}(\Sigma) = W_g(\Sigma), \quad 
				W_{g_\e}'(\Sigma) = \e^3 W_{g}'(\Sigma).
			\]
		Therefore, we may see that $\Sigma$ is a Willmore type surface in $(U_0,g)$ 
		if and only if it is so is in $(U_0,g_\e)$. 
	
	\item[(ii)] 
		The exponential map $\exp^{g_\e}_P $ is defined in 
		$B_{\e^{-1} \rho_0}(0)$ and satisfies 
			\[
				\exp^{g_\e}_P(z) = \exp^g_{P}(\e z) \quad 
				{\rm for\ all} \ |z|_{g_0} \leq \e^{-1} \rho_0 \quad 
				{\rm and} \quad P \in U_0.
			\]
		Moreover, $g_{\e,P,\a \b}$ can be expanded as 
			\begin{equation}\label{31}
				g_{\e,P,\a \b} (y) = \delta_{\a \b} + \e^2 h^\e_{P, \a \b} (y)
				\quad {\rm for\ } |y|_{g_0} \leq \e^{-1} \rho_0
			\end{equation}
		and $g_{\e,P,\a \b}$ satisfies 
			\begin{equation}\label{32}
				|y|^{-2} | D_P^k (g_{\e,P,\a\b}(y) - \delta_{\a \b} )| 
				+ |y|^{-1} | D_P^k D_y g_{\e,P,\a\b}(y)| 
				+ \sum_{j=2}^\ell |D_P^k D_y^j g_{\e,P,\a\b}(y) | 
				\leq C_{k,\ell} \e^2
			\end{equation}
		for any $k,\ell \geq 0$ where $D_P, D_y$ stand for derivatives 
		in $P$ and $y$, respectively. 
	\end{enumerate}
Using the metric $g_\e$, we set 
	\[
		\mathcal{T}_{\e,V_0} := \{ \exp_{P}^{g_\e} (S^2) \ |\ P \in \overline{V_0} \}. 
	\]
Notice that $\Sigma \subset U_0$  for each $\Sigma \in \mathcal{T}_{\e,V_0}$, 
by the properties of $U_0$ and $V_0$. 
Due to the above expansion of $g_\e$, elements in $\mathcal{T}_{\e,V_0}$ 
are approximate solutions to our problem. Namely,

	\begin{lmm}\label{301}
For any $k , \ell \in \N$ and $j=0,1$, one has 
	\[
		\left\| D_\e^j D_P^k W_{g_\e}'( \exp_{P}^{g_\e} (S^2)) \right\|_{C^\ell(\Sigma)} 
		\leq C_{k,\ell} \e^{2-j} \quad {\rm for\ all}\ P \in \overline{V_0} 
		\  {\rm and}\ \e \in (0,\e_0)
	\]
where we regard $W_{g_\e}'(\Sigma)$ as a function satisfying 
$\delta W_{g_\e}(\Sigma) [\varphi] = \int_{\Sigma} W_{g_\e}'(\Sigma) \varphi d \sigma$.
	\end{lmm}

	\begin{proof}
Denote by $H_{\e,P}$, $A_{\e,P}$ and $\Ric_{\e,P}$
the mean curvature of $\exp_P^{g_\e}(S^2)$, the second fundamental form 
and the Ricci curvature in $(\R^3,g_{\e,P})$. 
Remark that the unit outer normal to $\exp^{g_\e}_P(S^2)$ is given by $n_0(q) = q$ 
($q \in S^2$) by Gauss' lemma. 
We use the notation $H_0$, $A_0$ and so on 
for those in the Euclidean space. 
Recall from Proposition \ref{202} that 
	\[
		W_{g_\e}'( \exp^{g_\e}_P (S^2) ) 
		= - \Delta_{g_{\e,P}} H_{\e,P} - |A_{\e,P} |^2_{g_{\e,P}} H_{\e,P} 
		- H_{\e,P} \Ric_{\e,P}( n_{\e,P}, n_{\e,P} ) + \frac{1}{2} H_{\e,P}^3. 
	\]
By \eqref{31} and \eqref{32}, we observe that 
	\begin{equation}\label{33}
		\begin{aligned} 
			&\left| D_\e^j D_P^k D_{y}^\ell \left( g_{\e,P}^{\a \b} - \delta^{\a \b} \right) \right|
			\leq C_{k,\ell} \e^{2-j}, &
			&\left| D_\e^j D_P^k D_y^\ell \Ric_{\e,P}  \right| \leq C_{k,\ell} \e^{2-j},
			\\
			& \left\| D_\e^j D_P^k ( A_{\e,P} - A_0 ) \right\|_{C^\ell(S^2)}
				\leq C_{k,\ell} \e^{2-j},
			& &
			\left\| D_\e^j D_P^k ( \Delta_{g_{\e,P}} - \Delta_{g_0}  ) \varphi \right\|_{C^\ell(S^2)} 
			\leq C_{k,\ell} \e^{2-j} \| \varphi \|_{C^{\ell + 2}(S^2)}
		\end{aligned}
	\end{equation}
for all $P \in \overline{V}_0$, $k, \ell \in \N$, $j=0,1$, $|y|_{g_0} \leq \e^{-1} \rho_0$ and 
$\varphi \in C^{\ell + 2}(S^2)$ where $C_{k,\ell}$ depends only on $k,\ell$. 
Since $W_{g_0}'(S^2) = 0$ and $H_{\e,P}, H_0 \in C^\infty(S^2)$, we obtain 
	\[
		\left\| D_\e^j D_P^k W_{g_\e}'( \exp_P^{g_\e} (S^2)  ) \right\|_{C^\ell(S^2)}
		= \left\| D_\e^j D_P^k \left( W_{g_\e}'( \exp_P^{g_\e} (S^2)  ) - W_{g_0}'(S^2) \right) 
		\right\|_{C^\ell(S^2)} 
		\leq C_{k,\ell} \e^{2-j},
	\]
which completes the proof. 
	\end{proof}

Next, we find a correction for each $\Sigma \in \mathcal{T}_{\e,V_0}$ 
such  that it satisfies the area constraint condition and it
solves the equation up to some finite dimensional subspace in $L^2(S^2)$.

		To this aim, recall that $n_0(q) = q$ is an outer unit normal to $S^2$ with respect to $g_{\e,P}$. For $\varphi \in C^{4,\a}(S^2)$ and $P \in \overline{V}_0$, we set 
	\[
		S^2_{\e,P}[\varphi] := 
		\left\{ \left(1 + \varphi(q) \right) q   \ |\ q \in S^2  \right\} \subset \R^3, \quad 
		\Sigma_{\e,P}[\varphi] := 
		\exp_P^{g_\e} ( S^2_{\e,P}[\varphi] ).
	\]
Since we are interested in small $\varphi \in C^{4,\a}(S^2)$, 
we pull back all geometric quantities of $S^2_{\e,P}[\varphi]$ 
(or $\Sigma_{\e,P}[\varphi]$) on $S^2$. 
We denote by $\bar{g}_{\e,P,\varphi}$ the pull back of the tangential metric 
of $S^2_{\e,P}[\varphi]$ on $S^2$ and by $n_{\e,P,\varphi}$ the outer unit normal. 
We also write $L^2_{\e,P,\varphi}(S^2)$ and 
$\la \cdot , \cdot \ra_{\e,P,\varphi}$ the $L^2$-space on $S^2$ 
with volume induced by $\bar{g}_{\e,P,\varphi}$ and its inner product. 
We use the notations $\bar{g}_{0,\varphi}, n_{0,\varphi}, \ldots$ for the corresponding quantities in  the case $\e=0$, i.e. the Euclidean space case.

		Next, we define the space 
$\mathcal{K}_{\e,P,\varphi} \subset L^2_{\e,P,\varphi}(S^2)$ 
corresponding to $\mathcal{K}_0$ in \eqref{21}. 
Recalling $Z_0 = H_0/2$, we set 
\begin{equation}\label{eq:Kepphi}
\mathcal{K}_{\e,P,\varphi} := 
		{\rm span}\, \left\{ H_{\e,P,\varphi}, Z_{1}, Z_2, Z_{3} \right\} 
		\subset L^2_{\e,P,\varphi}(S^2), 
\end{equation}
where $H_{\e,P,\varphi}$ stands for the mean curvature of $\Sigma_{\e,P}[\varphi]$. 

\noindent
Next, we orthonormalize $H_{\e,P,\varphi}$, $Z_1$, $Z_2$ and $Z_3$ 
in $L^2_{\e,P,\varphi}(S^2)$ as follows. 
We first apply a Graham--Schmidt orthogonalization to 
$Z_1,Z_2,Z_3$ in $L^2_{\e,P,\varphi}$ to obtain 
$Y_{1,\e,P,\varphi}, Y_{2,\e,P,\varphi}, Y_{3,\e,P,\varphi}.$
Finally, we get $Y_{0,\e, P,\varphi}$ from $H_{\e,P,\varphi}$ and 
$Y_{i,\e,P,\varphi}$ ($1 \leq i \leq 3$). 
Then we define the $L^2_{\e,P,\varphi}(S^2)$-projection to 
$(\mathcal{K}_{\e,P,\varphi})^{\perp}$ by 
	\[
		\Pi^\varphi_{\e,P} \psi := \psi 
		- \sum_{i=0}^3 \la \psi , Y_{i,\e,P,\varphi} \ra_{\e,P,\varphi} 
		Y_{i,\e,P,\varphi} : 
		L^2_{\e,P,\varphi} (S^2) \to (\mathcal{K}_{\e,P,\varphi})^{\perp}. 
	\]
We denote with $Y_{i,0,\varphi}$ and $\Pi^\varphi_0$ the corresponding quantities in the Euclidean space.

	\begin{lmm}\label{302}
Let $k \in \N$ and $\a \in (0,1)$. Then 
there exist $r_{1}, \e_{1}, C_{k} > 0$ such that the maps 
	\[
		(0,\e_1) \times \overline{V}_0 \times \overline{B_{r_{1},C^{4+k,\a}}(0)} 
		\ni (\e, P,\varphi)  
		\mapsto Y_{i,\e,P,\varphi} \in C^{2+k,\a}(S^2) \quad 
		(0 \leq i \leq 3)
	\]
are smooth, where $B_{r_{1},C^{4+k,\a}}(0)$ stands for 
a metric ball in $C^{4+k,\a}(S^2)$. Moreover, the following estimates hold: for $j=0,1$, 
	\begin{equation}\label{34}
		\begin{aligned}
			&\left\| D_\e^j \left( Y_{i,\e,P,\varphi} - Y_{i,0,\varphi} \right) \right\|_{C^{2+k,\a}(S^2)} 
			+ \left\| D_\e^j D_{P}  (Y_{i,\e,P,\varphi} - Y_{i,0,\varphi}) 
			\right\|_{\mathcal{L}(\R^{3}, C^{2+k,\a}(S^2))} \\
			& \quad + \left\| D_\e^j D_{\varphi}  (Y_{i,\e,P,\varphi} - Y_{i,0,\varphi}) 
			\right\|_{\mathcal{L}(C^{4+k,\a}(S^2), C^{2+k,\a}(S^2))} 
			+ \left\| D_\e^j D_{P,\varphi}^2 (Y_{i,\e,P,\varphi} - Y_{i,0,\varphi}) 
			\right\|_{\mathcal{L}^2(C^{4+k,\a}(S^2), C^{2+k,\a}(S^2))} 
			\\ 
			&\quad \qquad 
			\leq C_k \e^{2-j} \left( 1 + \| \varphi \|_{C^{4+k,\a}(S^2)} \right).
		\end{aligned}
	\end{equation}
	\end{lmm}

	\begin{proof}
Let $X_\varphi(q)$ be the position vector for $S^2[\varphi]$: 
	\begin{equation}\label{035}
		X_\varphi(q) := \left( 1 + \varphi(q) \right) q \quad 
		{\rm for}\ q \in S^2.
	\end{equation}
Next, we fix small $r_{1} > 0$ and $\e_1>0$ so that 
$S^2[\varphi]$ is diffeomorphic to $S^2$ and 
$\exp_P^{g_\e}( S^2_{\e,P}[\varphi] )$ can be defined for 
all $\e \in [0,\e_1)$, $P \in \overline{V}_0$ and $\varphi \in C^{4+k,\a}(S^2,\R)$ with 
$\| \varphi \|_{C^{4+k,\a}(S^2)} \leq r_1$. 
Hereafter, we only deal with $\varphi \in \overline{B_{r_{1},C^{4+k,\a}}(0)}$ 
and $\e \in [0,\e_1)$. 
Then it is easily seen that the map 
	\[
		\varphi \mapsto X_{\varphi} :
		\overline{B_{r_{1},C^{4+k,\a}}(0)}
		\to C^{4+k,\a}(S^2,\R^3)
	\]
is smooth. 
Hence, we observe that the maps
	\[
		\begin{aligned}
			& (\e, P,\varphi) \mapsto g_{\e,P} (X_{\varphi}) &:& & &
			(0,\e_1) \times \overline{V}_0 \times \overline{B_{r_{1},C^{4+k,\a}}(0)}
			\to C^{4+k,\a}(S^2, (T\R^3)^\ast \otimes (T\R^3)^\ast  ), 
			\\
			& (\e, P,\varphi) \mapsto \bar{g}_{\e,P,\varphi}, \ \bar{g}_{0,\varphi} 
			&:& & & (0,\e_1) \times \overline{V}_0 \times \overline{B_{r_{1},C^{4+k,\a}}(0)} 
			\to C^{3+k,\a}(S^2, (TS^2)^\ast \otimes (TS^2)^\ast  ),
			\\
			& (\e, P,\varphi) \mapsto n_{\e,P,\varphi}, \ n_{0,\varphi}
			&:& & &
			(0,\e_1) \times \overline{V}_0 \times \overline{B_{r_{1},C^{4+k,\a}}(0)} 
			\to C^{3+k,\a} (S^2, \R^3),
			\\
			& (\e, P,\varphi) \mapsto H_{\e,P,\varphi}, \ H_{0,\varphi} 
			&:& & &
			(0,\e_1) \times \overline{V}_0 \times \overline{B_{r_{1},C^{4+k,\a}}(0)} 
			\to C^{2+k,\a}(S^2,\R)
		\end{aligned}
	\]
are smooth. Moreover, by \eqref{31}, we have
	\[
		\left\| D_\e^j D_{P}^\ell D_{\varphi}^m 
		\left( g_{\e,P} \left( X_{\varphi} \right) - g_0 \left( X_\varphi \right)  \right)
		\right\|_{ \mathcal{L}^m \left( C^{4+k,\a} (S^2) , C^{4+k,\a} (S^2) \right) } 
		\leq C_{k,\ell,m} \e^{2-j} \left( 1 +  \| \varphi \|_{C^{4+k,\a}} \right)
	\]
for $j=0,1$ and $\ell,m = 0,1,2$. Thus we obtain 
		\begin{align}
			&\left\| D_\e^j D_{P}^\ell D_\varphi^m 
			\left( \bar{g}_{\e,P,\varphi} - \bar{g}_{0,\varphi} \right) 
			\right\|_{ C^{3+k,\a}(S^2)  } 
			+ \left\| D_\e^j D_{P}^\ell D_\varphi^m  
			\left( n_{\e,P,\varphi} - n_{0,\varphi} \right) 
			\right\|_{ C^{3+k,\a}(S^2)  } \nonumber
			\\
			& \quad \qquad + \left\| D_\e^j D_{P}^\ell D_\varphi^m  
			\left( H_{\e,P,\varphi} - H_{0,\varphi} \right) 
			\right\|_{  C^{2+k,\a}(S^2) } 
			\leq C_k \e^{2-j} \left( 1 + \| \varphi \|_{C^{4+k,\a}(S^2)} \right) \label{eq:estgHn}
		\end{align}
for $j=0,1$ and $\ell,m = 0,1,2$, where we used a shorthand notation to denote the norms in the space of (multi)-linear operators. Now it is easily seen that \eqref{34} holds and we complete the proof. 
	\end{proof}

		 We next find a correction term for each element in $\mathcal{T}_{\e,V_0}$ 
so that the resulting surfaces solve the (area-constrained) Willmore equation up to an error in the
finite dimensional subspace $\mathcal{K}_{\e,P,\varphi}$ in \eqref{eq:Kepphi}.

	\begin{prp}\label{303}
Let $\a \in (0,1)$. 
There exist $C > 0$ and $\e_{2} > 0$ so that 
for every $\e \in (0,\e_{2})$ and $P \in \overline{V}_0$, 
there exists a unique $\varphi_{\e,P} \in C^{5,\a}(S^2)$ satisfying 
	\[
		\begin{aligned}
			&\emph{(i)} \quad 
			W_{g_\e}'( \Sigma_{\e,P}[\varphi_{\e,P}] ) 
			= \beta_0 H_{\e,P,\varphi_{\e,P}}
			+ \sum_{i=1}^3 \beta_i Z_{i}; 
			& &
			\emph{(ii)} \quad 
			| \Sigma_{\e,P}[\varphi_{\e,P}] |_{g_\e} = 4 \pi;
			\\
			&
			\emph{(iii)} \quad \left\la \varphi_{\e,P} , Y_{i,\e,P} 
			\right\ra_{\e,P,\varphi_{\e,P}} = 0 \ (1 \leq i \leq 3) ;
			& &
			\emph{(iv)} \quad 
			\| \varphi_{\e,P} \|_{C^{5,\a}(S^2)} \leq C \e^2 
		\end{aligned}
	\]
for some real numbers $\beta_0,\ldots, \beta_3$ where $|\Sigma|_{g_\e}$ 
denotes the area of $\Sigma$ in $g_\e$. 
Moreover, the map 
$(\e,P) \mapsto \varphi_{\e,P} : (0,\e_{2})\times \overline{V}_0 \to C^{5,\a}(S^2)$ is smooth and 
satisfies 
	\[
		\| D_P \varphi_{\e,P} \|_{\mathcal{L}(\R^3,C^{5,\a}(S^2))} 
		+ \| D_P^2 \varphi_{\e,P} \|_{ \mathcal{L}^2(\R^3,C^{5,\a}(S^2)  ) } 
		\leq C \e^2, \quad \| D_\e \varphi_{\e,P} \|_{C^{5,\a}(S^2))} \leq  C \e 
		\quad {\rm for\ all} \ \e \in (0,\e_2).
	\]
	\end{prp}

	\begin{proof}
Define a map 
$G (\e, P,\varphi) : (0,\e_1) \times \overline{V}_0 \times \overline{B_{r_{1},C^{5,\a}}(0)} 
\to C^{1,\a}(S^2,\R) $ by 
	\[
		G (\e, P,\varphi) := 
		\Pi^\varphi_{\e,P} \left( W_{g_\e}' \left( \Sigma_{\e,P}[\varphi] \right) \right) 
		+ \left( \left| \Sigma_{\e,P}[\varphi] \right|_{g_\e} - 4 \pi  \right) H_{\e,P,\varphi} 
		+ \sum_{i=1}^3 \left\la Y_{i,\e,P,\varphi} , \varphi \right\ra_{\e,P,\varphi} 
		Y_{i,\e,P,\varphi}.
	\]
By definition of $\mathcal{K}_{\e,P,\varphi}$, $Y_{j,\e,P,\varphi}$ and 
$\Pi^\varphi_{\e,P}$, to obtain the properties (i)--(iii)
it is enough to find $\varphi_{\e,P}$ satisfying 
$G (\e, P,\varphi_{\e,P}) = 0$.

	For this purpose, we show the existence of $\e_2>0$ and $r_2>0$ so that 
$D_\varphi G(\e, P,\varphi) : C^{5,\a}(S^2) \to C^{1,\a}(S^2)$ is invertible for all 
$(\e, P,\varphi) \in (0,\e_2) \times \overline{V}_0 \times \overline{B_{r_{2},C^{5,\a}}(0)}$. 
We first remark that $G$ is smooth in $\e,P$ and $\varphi$. 
Moreover, by Lemma \ref{302} together with \eqref{31}--\eqref{33}, 
we have the following estimates 
	\begin{equation}\label{35}
		\sum_{k,\ell =0}^2 \left\| D_\e^j D_{P}^k D_\varphi^\ell \left( G (\e, P,\varphi) - G_0(\varphi) \right) 
		\right\|_{C^{1,\a}(S^2)} \leq C \e^{2-j} \left( 1 + \| \varphi \|_{C^{5,\a}(S^2)} \right)
	\end{equation}
for each $j=0,1$ and $(\e, P,\varphi) \in (0,\e_1) \times \overline{V}_0 \times \overline{B_{r_{1},C^{5,\a}}(0)}$. 
Here $G_0(\varphi)$ is the corresponding map 
in the Euclidean space. 
Thanks to \eqref{35}, it suffices to show that 
$D_\varphi G_0(0)$ is invertible.

		For this, we recall from Proposition \ref{202} and 
the comments below it that 
	\[
		D_\varphi W_{g_0}'(S^2[\varphi]) \big|_{\varphi =0} [\psi] 
		= \frac{d}{dt} W_{g_0}'(S^2[t \psi]) \big|_{t=0} 
		= \tilde{L}_0 \psi 
		= \Delta (\Delta + 2) \psi
	\]
for $\psi \in C^{5,\a}(S^2)$. 
Hence, since $|S^2|_{g_0} = 4\pi$ and $W_{g_0}'(S^2) = 0$, we obtain 
	\[
		D_\varphi G_0(0) [\psi] 
		= \Pi^0_0  \tilde{L}_0 \psi  + 
		\la H_{0} , \psi \ra_{L^2(S^2)} H_0 
		+ \sum_{i=1}^3 \la Y_{i,0} , \psi \ra_{L^2(S^2)} Y_{i,0}. 
	\]
Noting that 
	\[
		\la \varphi , \tilde{L}_0 \psi \ra_{L^2(S^2)} = 
		\la \tilde{L}_0 \varphi , \psi \ra_{L^2(S^2)}
	\]
holds for any $\varphi,\psi \in C^4(S^2)$ and that 
	\[
		{\rm Ker}\, \tilde{L}_0 = \mathcal{K}_0 
		= {\rm span}\, \{ H_0, Y_{1,0}, Y_{2,0}, Y_{3,0} \}, 
	\]
we have 
$\Pi^0_0 \tilde{L}_0 \psi = \tilde{L}_0 \psi$ and 
	\begin{equation}\label{36}
		D_\varphi G_0(0) [\psi] 
		= \tilde{L}_0 \psi  + 
		\la H_{0} , \psi \ra_{L^2(S^2)} H_0 
		+ \sum_{i=1}^3 \la Y_{i,0} , \psi \ra_{L^2(S^2)} Y_{i,0}. 
	\end{equation}
Moreover, by  Fredholm's alternative and 
 elliptic regularity theory, we notice that 
	\[
		\tilde{L}_0 \left( C^{5,\a} ( S^2 )  \right) 
		= C^{1,\a} (S^2) \cap \mathcal{K}_0^\perp.
	\]
Thus it follows from \eqref{36} and Schauder's estimates that $D_\varphi G_0(0)$ 
is invertible and by \eqref{35}, we may find 
$\e_2>0$ and $r_2>0$ satisfying the desired property.

		Now, for $\e \in [0,\e_2)$, the Inverse Mapping Theorem 
ensures the existence of neighborhoods $U_{1,\e,P} \subset C^{5,\a}(S^2)$ of $0$ 
and $U_{2,\e,P} \subset C^{1,\a}(S^2)$ of $G(\e,P,0)$ such that 
$G(\e,P, \cdot) : U_{1,\e,P} \to U_{2,\e,P}$ is  a diffeomorphism. 
Furthermore, by \eqref{35} and the proof of the Inverse Mapping Theorem 
(see Lang \cite[Theorem 3.1 in Chapter XVIII]{LANG}), 
shrinking $r_2>0$ if necessary, we may assume that 
	\[
		\overline{B_{r_2,C^{5,\a}(S^2)}(0)} \subset U_{1,\e,P}, \quad  \qquad
		\overline{B_{2r_2,C^{1,\a}(S^2)}(G(\e,P,0))} \subset U_{2,\e,P}
	\]
for all $(\e,P) \in (0,\e_2) \times \overline{V}_0$. 
Noting that $\| G (\e,P,0) \|_{C^{1,\a}(0)} \leq C \e^2$ holds 
due to Lemma \ref{301} and \eqref{31}--\eqref{32}, 
shrinking $\e_2>0$ enough, we have 
	\[
		\overline{B_{r_2,C^{1,\a}(S^2)}(0)} \subset 
		\overline{B_{2r_2,C^{1,\a}(S^2)}(G(\e,P,0))}. 
	\]
for each $(\e,P) \in (0,\e_2) \times \overline{V}_0$. 
In particular, $ 0 \in U_{2,\e,P}$ holds and setting 
$\varphi_{\e,P} := \left( G(\e,P,\cdot) \right)^{-1} (0) $, we see that 
the properties (i)--(iii) hold. 
\\From \eqref{35} and $G(\e,P_\e,\varphi_{\e,P}) = 0$, we get $\|G_{0}(\varphi_{\e,P})\|_{C^{1,\a}(S^2)}\leq C \e^{2}$. Thus, by the invertibility of $G_{0}$, also (iv) holds.

		The smoothness of $\varphi_{\e,P}$ in $(\e,P)$ follows from that of 
$G$ and the Implicit Function Theorem. 
The estimates on $D_P^k \varphi_{\e,P}$, $k=1,2$, (resp. on $D_\e \varphi_{\e,P}$) follow from 
differentiating the equation 
$G(\e,P, \varphi_{\e,P}) = 0$ in $P$ (resp. in $\e$) and using that 
$\| (D_P^k G) (\e, P,\varphi_{\e,P}) \|_{C^{1,\a}(S^2)} \leq C\e^2$ for $k=1,2$ 
due to Lemmas \ref{301} and \ref{302}, \eqref{31}--\eqref{33}, \eqref{35} and the fact that $D_\varphi G_0(0)$ is invertible 
(resp. using that $\| (D_\e G) (\e,P,\varphi_{\e,P}) \|_{C^{1,\a}(S^2)} \leq C\e$). 
Hence, the proof of Proposition \ref{303} is complete. 
	\end{proof}

	Recalling the function $\varphi_{\e,P}$ given by Proposition \ref{303}, we set 
	\[
		\Phi_\e(P) := W_{g_\e} ( \Sigma_{\e,P} [\varphi_{\e,P}]  )  
		\in C^2(\overline{V}_0,\R).
	\]

	\begin{prp}\label{304}
There exists  $\e_3>0$ such that if $\e \in (0,\e_3)$ and 
$P_\e \in V_0$ is a critical point of $\Phi_\e$, 
then $\Sigma_{\e,P_\e}[\varphi_{\e,P_\e}]$ satisfies 
the area-constrained Willmore equation, namely, 
$\beta_1=\beta_2=\beta_3 = 0$ hold in 
Proposition \emph{\ref{303} (i)}. 
	\end{prp}

	\begin{proof}
We first remark that the criticality of $\Phi_\e(P)$ is independent of the
choices of charts; we will use  normal coordinates with respect to the metric  $g_\e$  centered  at $P$. 
We  will use the same notation as in the proof of Lemma \ref{302}: in particular recall \eqref{035}.

	Assume that $P_\e \in V_0$ is a critical point of $\Phi_\e$ and 
let $(U,\Psi)$ be a normal coordinate system centered  at $P_\e$. For $P \in U$ 
with $z = \Psi(P)$, a position vector for $\Sigma_{\e,P}[\varphi]$ in 
$(U,\Psi)$ has the form 
	\[
		\tilde{X}_{\e,P,\varphi}(q) 
		= \mathcal{X}_\e(1;z, T_\e(z) ( X_{\e,P_\e,\varphi}(q)  )) \quad 
		{\rm for} \ q \in S^2,
	\]
where $T_\e(z) : \R^3 \to \R^3$ is a linear transformation with $T_\e(0) = {\rm Id}$, 
$\mathcal{X}_\e$ a solution of 
	\[
		\frac{d^2 \mathcal{X}_\e^\a}{dt^2} 
		+ \Gamma^\a_{\e,\b \gamma} (\mathcal{X}_\e) 
		\frac{d \mathcal{X}_\e^\b}{dt} \frac{d \mathcal{X}_\e^\gamma}{dt} = 0, \quad 
		\left( \mathcal{X}_\e(0;z,v) , \frac{d \mathcal{X}_\e}{dt} (0;z,v) \right) 
		= (z,v) \in \R^6
	\]
and $\Gamma^\a_{\e,\b \gamma}$ stand for the Christoffel symbols 
of $(M,g_\e)$ in the coordinate system $(U,\Psi)$. 
Using \eqref{31} and \eqref{32}, we may observe that 
for any $k,\ell \geq 0$, 
	\begin{equation}\label{38}
		\begin{aligned}
			&\left\| D_z^{k+1} T_\e (z) \right\|_{L^\infty} \leq C_{k} \e^2, 
			& & \left\| \Gamma^\a_{\e,\b \gamma} \right\|_{C^{k}} 
			\leq C_{k} \e^2,
			\\
			& \mathcal{X}_\e(1;z,v) = z + v + R_{\e}(z,v), 
			& &\left\| D_z^{k} D_v^\ell R_{\e} (z,v) \right\|_{L^\infty} 
			\leq C_{k,\ell} \e^2. 
		\end{aligned}
	\end{equation}
Therefore, we have 
	\[
		\tilde{X}_{\e,P,\varphi}(q) 
		= z + X_{\e,P_\e,\varphi}(q) + \tilde{R}_\e (z, X_{\e,P_\e,\varphi}(q) )
	\]
where $\tilde{R}$ satisfies the same estimate as $R_\e$ in \eqref{38}.

	We now  differentiate $\tilde{X}_{\e,P,\varphi_{\e,P}}(q)$ in $z^i$ 
($1 \leq i \leq 3$). 
By Proposition \ref{303} and \eqref{38}, we obtain 
	\[
		\frac{\partial \tilde{X}_{\e,P,\varphi_{\e,P}}}{\partial z^i} 
		= \mathbf{e}_i + O_{C^{5,\a}}(\e^2)
	\]
where $\| O_{C^{5,\a}}(\e^2) \|_{C^{5,\a}(S^2)} \leq C\e^2$. 
Moreover, one has 
	\[
		n_{\e,P,\varphi_{\e,P}} - n_0 = O_{C^{4,\a}}(\e^2).
	\]
Recalling the definition of $Z_i(q)$ in Section \ref{2000} and 
setting 
	\begin{equation}\label{eq:defPsiieP}
		\psi_{i,\e,P} (q) := g_\e ( \tilde{X}_{\e,P, \varphi_{\e,P}} (q) ) 
		\left( \frac{\partial \tilde{X}_{\e,P,\varphi_{\e,P}} (q) }{\partial z^i} , 
		n_{\e,P,\varphi_{\e,P}} (q) \right), 
	\end{equation}
we find that $\psi_{i,\e,P} - Z_i = O_{C^{4,\a}}(\e^2)$. 
Finally, differentiating $|\Sigma_{\e,P} [\varphi_{\e,P}] |_{g_\e} = 4 \pi$ in 
$z^i$, we get 
	\[
		\left\la H_{\e,P,\varphi_{\e,P}} , \psi_{i,\e,P} \right\ra_{\e,P,\varphi_{\e,P}} 
		= 0. 
	\]

	Now, by $\partial_{z^i} \Phi_\e(P_\e) = 0$ and 
$\la Z_i , Z_j \ra_{L^2(S^2)} = \delta_{ij} \| Z_i \|_{L^2(S^2)}^2$ due to \eqref{22}, we have 
	\[
		\begin{aligned}
			0 
			= \partial_{z^i} \Phi_\e(P_\e) 
			&= \left\la W_{g_\e}'( \Sigma_{\e,P_\e}[\varphi_{\e,P_\e}] ) , \psi_{i,\e,P_\e} 
			\right\ra_{\e,P_\e,\varphi_{\e,P_\e}} 
			\\
			&= \sum_{j=1}^{3} \beta_j \left\la Z_j , \psi_{i,\e,P_\e} 
			\right\ra_{\e,P_\e,\varphi_{\e,P_\e}} 
			= \sum_{j=1}^3 \beta_j
			\left( \delta_{ij} \| Z_i \|_{L^2(S^2)}^2 + O(\e^2)  \right).
		\end{aligned}
	\]
Thus there exists  $\e_3>0$ such that if $\e \in (0,\e_3)$, then we have 
$\beta_1 = \beta_2 = \beta_3 = 0$ and Proposition \ref{304} holds. 
	\end{proof}
	
	\begin{rem}\label{rem:PhieGlobal}
	If $(M,g)$ is a 3-dimensional compact Riemannian manifold without boundary, then  we can define globally the reduced functional $\Phi_{\e}:M\to \R$ as
	 \[
		\Phi_\e(P) := W_{g_\e} ( \Sigma_{\e,P} [\varphi_{\e,P}]  ) \in C^2(M,\R) \quad  \forall \e\in (0,\bar{\e}],\text{ for some $\bar{\e}=\bar{\e}(M)>0$}.
	\]
	Moreover $P_{\e}$ is a critical point of $\Phi_{\e}$ if and only if the perturbed geodesic sphere  $\Sigma_{\e,P} [\varphi_{\e,P}]$ is an area-constrained Willmore surface of area $4\pi \e^{2}$. 
	\\Indeed, for every $P\in M$ we can find a neighborhood $U_{P}\ni P$ and $\e_{P}>0$ such that  $\Phi_\e:U_{P}\to \R$ is well defined as above for every $\e\in (0,\e_{P}]$.  Note that if two neighborhoods overlap, then the corresponding definitions of
	$\Phi_{\e}$ agree thanks to the uniqueness of $\varphi_{\e,P}$ in Proposition \ref{303}. By the compactness of $M$ we can then find $P_{1},\dots, P_{N}$ so that $M=\cup_{i=1}^{N} U_{P_{i}}$. Hence, setting $\bar{\e}(M)=\min\{\e_{P_{1}}, \ldots, \e_{P_{N}}\}>0$
	 and patching  the local definitions of $\Phi_{\e}$,  the claim is proved.
	\end{rem}


\section{Concentration of area-constrained Willmore spheres}


The goal of this section is to prove the next result.

\begin{thm}\label{101}
Let $(M,g)$ be a 3-dimensional Riemannian manifold and $\Sc$ denote the scalar 
curvature of $M$. Assume that $P_0 \in M$ is a non-degenerate critical point of 
$\Sc$. Then there exists  $\e_0>0$ such that for each $\e \in (0,\e_0)$, 
there exists an area-constrained Willmore  sphere $\Sigma_\e \subset M$ 
with $|\Sigma_\e|_g = 4 \pi \e^2$ such that 
$\Sigma_\e$ concentrates at $P_0$.  More precisely $\Sigma_{\e}$ is the normal graph of a function $\varphi_{\e}\in C^{5,\alpha}(S^{2})$ over a geodesic sphere centered in $P_{\e}$ satisfying:
\begin{equation}\label{eq:propSe}
\Sigma_{\e}:=\Sigma_{\e,P_{\e}}[\varphi_{\e}] := \exp_{P_{\e}}^{g_\e} ( S^2_{\e,P_{\e}}[\varphi_{\e}] ), \quad \|\varphi_{\e}\|_{C^{5,\alpha}}\leq C\e^{2}, 
\end{equation}
for some constant $C=C(P_{0})>0$ independent  of $\e$.  
Moreover, if  the index of $P_{0}$ as a  critical point of $\Sc$ is equal to $3-k$ 
, then each surface $\Sigma_\varepsilon$ is an area-constrained critical point of $W$ of  index $k$.
\end{thm}

In the proof of Theorem  \ref{101} we will use the next lemma.

	\begin{lmm}\label{401}
There exist $\e_3>0$ and $C>0$ such that 
if $\e \in (0,\e_3)$ and a function $\Psi(P) \in C^2(\overline{V}_0,\R)$ satisfies 
	\[
		\frac{1}{\e^2} 
		\left\| \Psi - 16 \pi + \frac{8\pi}{3} \e^2 \Sc
		\right\|_{C^2 (\overline{V}_0) } \leq C \e,
	\]
then $\Psi$ has a unique critical point $P_{\e} \in \overline{V}_0$. 
Moreover, as $\e \to 0$, we have $P_{\e} \to P_0$. 
	\end{lmm}

	\begin{proof}
We first recall that by the choice of $V_0$, 
$P_0$ is the unique critical point of $\Sc$ in $\overline{V}_0$ and 
$\Hess (\Sc)$ invertible on $\overline{V}_0$. 
Then it is easily seen that 
for sufficiently small $\zeta_0>0$, 
if $\psi \in C^2(\overline{V}_0,\R)$ satisfies 
$\| \psi - \Sc \|_{C^2(\overline{V}_0)} \leq \zeta_0$, 
then $\Hess (\psi)$ is invertible on $\overline{V}_0$ and 
$\psi$ has a unique critical point in $\overline{V}_0$. 
Setting $\psi_{\e}(P) := \e^{-2} (\Psi(P) - 16 \pi )$,
 note that $\| \psi_\e - \frac{8\pi}{3}\Sc \|_{C^2(\overline{V}_0)} \leq C \e$  and
$D_P^k \psi_\e = \e^{-2}D_P^k \Psi$ for $k=1,2$. Thus,
for sufficiently small $\e>0$, 
$\Psi$ has a unique critical point $P_\e \in \overline{V}_0$. 
Since $\psi_\e \to \Sc$ in the $C^2$-sense and $P_0$ is the unique 
critical point of $\Sc$, we have 
$P_\e \to P_0$ as $\e \to 0$. 
	\end{proof}

	\begin{proof}[Proof of Theorem \ref{101}]
By Proposition \ref{304} and Lemma \ref{401}, it is enough to prove that 
	\begin{equation}\label{41}
		\frac{1}{\e^2} 
		\left\| \Phi_\e - 16\pi + \frac{8\pi}{3} \e^{2} \Sc \right\|_{C^2(\overline{V}_0)} \leq C \e
	\end{equation}
where $C>0$ is independent of $\e$. 
To this aim, we decompose $\Phi_\e - 16\pi + \frac{8\pi}{3} \e^{2} \Sc$ as follows: 
	\[
		\Phi_\e(P) - 16 \pi + \frac{8 \pi}{3} \e^{2} \Sc_P 
		= \left( \Phi_\e(P) - W_{g_\e} ( \Sigma_{\e,P}[0] ) \right) 
		+ \left( W_{g_\e} ( \Sigma_{\e,P}[0] ) - 16 \pi + 
		\frac{8 \pi}{3} \e^{2} \Sc_P  \right). 
	\]

		For the latter part, we notice that 
in the $C^0$-sense, we have 
	\begin{equation}\label{42}
		\frac{1}{\e^2}\left\| W_{g_{\e}}( S^2_{\e,(\cdot)}[0] ) 
		- 16 \pi + \frac{8\pi}{3} \e^{2} \Sc_{(\cdot)}  \right\|_{L^\infty(\overline{V}_0)}
		\leq C \e
	\end{equation}
where $C>0$ is independent of $\e$. 
For instance, see \cite{M-10,LM-10,IMM-1}. 
For the $C^2$-estimate of \eqref{42}, we provide here a self-contained argument; later, in Lemma \ref{lmm:geomExp}, we will give sharper estimates building on top of \cite{M-JGA}. 
Recalling the expansion of $g_\e$ in \eqref{31} and \eqref{32}, 
 setting $t = \e^2$ and 
	\[
		g_{t,P,\a \b} (y) := \delta_{\a \b} + t h_{P,\a \b}^{\sqrt{t}} (y), 
	\]
we can check that 
$t \mapsto D_{(P,y)}^k g_{t,P,\a \b}$ is of class $C^{1,1/2}$ at $t=0$ 
for each $k \in \N$. 
Hence, writing $W_{t,P}$ for the Willmore functional with respect to the metric $g_{t,P}$, 
we observe that the map $t \mapsto D_P^k W_{t,P}( S^2 )$ 
is also of class $C^{1,1/2}$ in $t$. 
Thus we have  
	\[
		\begin{aligned}
			& D_P^k 
			\left( W_{t,P}( S^2 ) - W_{0,P} ( S^2 )
			- \frac{\partial}{\partial s} W_{s,P} ( S^2 ) 
			\Big|_{s=0} t 
			  \right) 
			  \\
		= &  D_P^k \int_{0}^t 
		\left( \frac{\partial }{\partial s} W_{s,P}( S^2 ) 
		- \frac{\partial}{\partial s} W_{s,P} ( S^2 ) \Big|_{s=0} 
		\right) \rd s 
		= O(t^{3/2}). 
		\end{aligned}
	\]
>From the $C^0$-estimate, we deduce that 
	\[
		\frac{\partial}{\partial s} W_{s,P} ( S^2 ) \Big|_{s=0}
		=- \frac{8 \pi}{3} \Sc_P.
	\]
Noting that $W_{0,P} (S^2 ) = 16 \pi$ and $t=\e^2$, it follows that 
	\begin{equation}\label{43}
		\frac{1}{\e^2}\left\| W_{g_{\e}}( S^2_{\e,(\cdot)}[0] ) 
		- 16 \pi + \frac{8\pi}{3}\e^{2} \Sc_{(\cdot)}  \right\|_{C^2(\overline{V}_0)}
		\leq C \e.
	\end{equation}

In order to conclude the proof of Theorem \ref{101}, we are left with showing:
	\begin{equation}\label{44}
		\frac{1}{\e^2}\left\| \Phi_\e(\cdot) -  W_{g_{\e}}( S^2_{\e,({\cdot})}[0] ) 
		\right\|_{C^2(\overline{V}_0)}
		\leq C \e.
	\end{equation}
For this purpose, let us denote by $X_{\e,P,s}(q)$ a position vector 
for $S^2_{\e,P} [s \varphi_{\e,P}]$ with $ 0 \leq s \leq 1$, namely, 
	\[
		X_{\e,P,s}(q) := \left(1 + s \varphi_{\e,P} (q) \right) q \quad 
		{\rm for\ } q \in S^2, \ 0 \leq s \leq 1.
	\]
We also write $n_{\e,P,s}(q)$ ($q \in S^2$) for 
the outer unit normal to $S^2_{\e,P}[s \varphi_{\e,P}]$. Recall that 
	\[
		\left\| D_P^k n_{\e,P,s} \right\|_{C^{4,\a}(S^2)} \leq C \e^2
		\quad {\rm for}\ 0 \leq s \leq 1, \ k=0,1,2, 
	\]
where $C$ is independent of $s$ and $\e$. 
Thus setting 
	\[
		\psi_{\e,P,s}(q) := g_{\e,P}(X_{\e,P,s}(q)) [ \varphi_{\e,P} (q) , n_{\e,P,s} (q) ]
		\quad {\rm for}\ q \in S^2
	\]
and recalling the estimates of $\varphi_{\e,P}$ in Proposition \ref{303}, 
it follows that 
	\begin{equation}\label{eq:psi}
		\| D_P^k \psi_{\e,P,s} \|_{C^{4,\a}(S^2)} \leq C\e^2 
		\quad {\rm for}\ k=0,1,2, \ 0 \leq s \leq 1.
	\end{equation}
Furthermore, since 
$\| D_P^k g_{\e,P} \|_{C^\ell} \leq C_{\ell} \e^2$ holds 
for every $k=1,2$ and $\ell \in \N$, 
the estimates for $\varphi_{\e,P}$ and a similar argument 
to the proof of Lemma \ref{301} imply 
	\[
		\left\| D_P^k W_{g_\e}'( S_{\e,P} [s \varphi_{\e,P}]  )  
		\right\|_{C^0(S^2)} \leq C \e^2 \quad 
		{\rm for}\ k=0,1,2, \ 0 \leq s \leq 1.
	\]
Now from 
	\begin{equation}\label{eq:diff-e0}
		\begin{aligned}
			D_P^k \left(\Phi_\e(P) - W_{g_\e} (S_{\e,P}^2[0] ) \right)
			&= D_P^k\left( W_{g_\e} (S_{\e,P} [\varphi_{\e,P}] )  
			- W_{g_\e} (S_{\e,P}^2[0] ) \right)
			\\
			& = \int_0^1 D_P^k \frac{d}{d s} 
			W_{g_\e}( S_{\e,P} [s \varphi_{\e,P}] ) d s 
			\\
			& = \int_0^1 D_P^k \left(
			W_{g_\e}' ( S_{\e,P} [s \varphi_{\e,P}] ) [ \psi_{\e,P,s} ]\right) ds,
		\end{aligned}
	\end{equation}
it follows that 
	\begin{equation}\label{eq:DkPhiWS}
		\left\| D_P^k \left(\Phi_\e(\cdot) - W_{g_\e} (S_{\e,(\cdot)}^2[0] ) \right)
		 \right\|_{L^\infty(\overline{V}_0)} 
		 \leq C \e^4 
	\end{equation}
for $k=0,1,2$. Hence \eqref{44} holds and the claim  \eqref{41} is a consequence of \eqref{43} and  \eqref{44}.
The combination of  Proposition \ref{304} and Lemma \ref{401} gives directly all the claims of Theorem \ref{101}; we just briefly add some details regarding the index identity.  
\\Note that the indexes of $P_{\e}$ as a critical point of $\Phi_{\e}$  and of $-\Sc$ agree thanks to \eqref{41}; moreover, since  $W_{g_0}''(S^{2})$ is positive definite on the orthogonal complement to its kernel (made of constant and affine functions) and $\|\varphi_{\e,P_{\e}}\|_{C^{5,\alpha}(S^{2})}\leq C\e^{2}$, it holds that  $W_{\g_\e}''(\Sigma_{\e})$ is positive-definite on the $L^{2}$-orthogonal complement to $\{\psi_{i,\e,P_{\e}}\}_{i=1,2,3}$ defined in \eqref{eq:defPsiieP}. Observing that the index of $W_{g_\e}''(\Sigma_{\e})$ in the direction of the span of $\{\psi_{i,\e,P_{\e}}\}_{i=1,2,3}$ coincides with the index of $P_{\e}$ as a critical point of $\Phi_{\e}$ and that the only missing direction is fixed by area-constraint, the claim on the index identity follows.
	\end{proof}
	
\bigskip

\noindent
We next prove  multiplicity  of area-constrained Willmore spheres for prescribed (small) area.

\begin{proof}[Proof of Theorem \ref{thm:multS}]
Thanks to Remark \ref{rem:PhieGlobal}, if $M$ is closed, then  we can define globally the reduced functional $\Phi_{\e}:M\to \R$ as
	 \[
		\Phi_\e(P) := W_{g_\e} ( \Sigma_{\e,P} [\varphi_{\e,P}]  ) \in C^2(M,\R) \quad  \forall \e\in (0,\bar{\e}],\text{ for some $\bar{\e}=\bar{\e}(M)>0$}.
	\]
Moreover $P_{\e}$ is a critical point of $\Phi_{\e}$ if and only if the perturbed geodesic sphere  $\Sigma_{\e,P} [\varphi_{\e,P}]$ is an area-constrained Willmore surface of area $4\pi \e^{2}$. 
\\Note that if $P_{1}^{\e}\neq P_{2}^{\e}$ are distinct critical points of $\Phi_{\e}$, then the corresponding area-constrained Willmore surfaces $\Sigma_{\e,P_{1}^{\e}} [\varphi_{\e,P_{1}^{\e}}]$  and $\Sigma_{\e,P_{2}^{\e}} [\varphi_{\e,P_{2}^{\e}}]$ are also distinct since by (iii) in Proposition \ref{303} the graph function $\varphi_{\e,P}$ is $L^{2}$-orthogonal to the translations $\{Y_{j,\e,P}\}_{j=1,2,3}$. 
\\The claim of Theorem \ref{thm:multS} thus reduces to establish  multiplicity of the critical points of $\Phi_\e:M\to \R$. 
>From the Lusternik-Schnirelman theory (see for instance \cite[Theorem 1.15]{CLOT}), the number of critical points of a real valued $C^{2}$-function on $M$ is bounded below by $\Cat(M)+1$ and, for a closed 3-dimensional manifold, the value for $\Cat(M)$ is computable in terms of the fundamental group \cite[Corollary 4.2]{GG}: 
\begin{itemize}
\item $\Cat(M)=1$,  if $M$ is simply connected (i.e. if and only if  $M$ is diffeomorphic to $S^3$ by the recent proof of  Poincar\'e's conjecture);
\item $\Cat(M)=2$  if $\pi_1(M)$ is a nontrivial free group;
\item $\Cat(M)=3$  otherwise.
\end{itemize}
Therefore, Theorem \ref{thm:multS} then holds.
\end{proof}

\section{Foliation}\label{0005}

The goal of this section is to prove Theorem \ref{thm:foliation}, namely the existence and uniqueness of a foliation by area-constrained Willmore spheres of a neighborhood of a non-degenerate critical point $P_{0}$ of the scalar curvature.

Before proving Theorem \ref{thm:foliation} in detail, let us briefly discuss what is the main geometric extra difficulty in establishing that the area-constrained Willmore spheres $\Sigma_{\e}$ constructed in Theorem  \ref{101} form a foliation.
\\The main point is to show that the {\em centers} $P_{\e}$ of $\Sigma_{\e}$ converge fast enough, say at order $O(\e^{2})$, to $P_{0}$  when compared to the shrinking {\em radius} of the spheres (which is of order $O(\e)$).
This is best explained with an example: the round spheres $\Sigma_{\e}\subset \R^{3}$  in the Euclidean 3-dimensional space, of center $(\e,0,0)$ and radius $\e$ are clearly  (area-constrained) Willmore spheres concentrating at the origin $(0,0,0)$ but do not form a foliation (as they are not pairwise disjoint).
\\Showing that ${\rm d}_{g}(P_{0},P_{\e})=O(\e)$ is straightforward: just recall that $P_{\e}$ is a critical point of $\Phi_{\e}$ and combine \eqref{41} with the assumption that  $P_{0}$ is a non-degenerate critical point of $\Sc$ (so that there exists $C_{\Sc}>0$ with 
$
|\nabla \Sc(P)|\geq  C_{\Sc}\, {\rm d}_{g}(P_{0},P)$ near $P_{0}$). On the other hand,  the estimate ${\rm d}_{g}(P_{0},P_{\e})=O(\e^{2})$ requires more work.
\\ The rough idea is to exploit the symmetry/anti-symmetry of the terms in the geometric expansions  in order to show that the term of order $O(\e)$ vanishes. To this aim, we start by recalling the expressions of the terms involved in the Willmore equation on a small geodesic sphere, see for instance \cite[Section 3.1]{M-JGA}. For $P\in \overline{V}_0$, we set $\Sigma_{\e,P}^{0} := \exp_P^{g} ( \e S^2 )$ where $\e S^{2}\subset T_{P}M\simeq \R^{3}$ is the round sphere of radius $\e$ parametrized by $q\in S^{2}$.
Since in the arguments it will be enough to know whether a term is odd with respect to  the antipodal map  $q\mapsto-q$ of $S^{2}$, 
we will use the following shorthand notation:
	\[
		\text{$\cO:S^{2}\to \R$ will denote an arbitrary \emph{odd} smooth function, i.e. $\cO(-q)=-\cO(q)$;}
	\]
In order to keep the notation short, the functions $\cO$ will be allowed to vary from formula to formula and also within the same line; 
moreover $\cO$ will depend on $P$ smoothly with $D^k_P \cO (-q) = - D^k_P \cO(q)$, but be independent of the parameter $\e$.

\begin{lmm}\label{lmm:geomExp}
The following expansions hold: at $q \in S^2$ and $P \in \ov{V}_0$, 
\begin{align}
&H_{\Sigma_{\e,P}^{0}}= \frac{2}{\e} - \frac{\e}{3} \Ric_P \left( q, q \right) + \cO \e^2 + O_{C^2}(\e^3),
\label{eq:HSeP} \\
&d\sigma_{\Sigma_{\e,P}^{0}}= \e^2 \left( 1 - \frac{\e^2}{6} \Ric_P (q,q) + \cO \e^3 + O_{C^2}(\e^4) \right) 
d \sigma_{S^2},
\label{eq:dvolSigmaePH}
\\
& W_g \left( \Sigma_{\e,P}^{0} \right) 
= 16 \pi - \frac{8\pi}{3} \e^2 \Sc_P + O_{C^2}(\e^4),
\label{eq:WE}
\\
& \frac{\partial }{\partial \e} W_g \left(  \Sigma^0_{\e,P} \right) 
= - \frac{16\pi}{3} \e \Sc_P + O_{C^2} \left( \e^3 \right)
\label{eq:WEdiff}
\end{align}
where $d\sigma_{S^2}$ denotes the area element induced from the Euclidean metric and 
the terms $O_{C^2}(\e^k)$ satisfy $ \sum_{i=0}^2 \| D_P^i O_{C^2} (\e^k) \|_{L^\infty(S^2)} \leq C_0 \e^k$. 
\end{lmm}

\begin{proof}
We note that these results were essentially obtained in \cite{M-JGA}. 
In fact, 
using a local orthonormal frame $\{F_{P,1}, F_{P,2} , F_{P,3} \}_{P \in \ov{V}_0}$ as in the beginning of section \ref{3000}, 
$(\exp_P)^\ast g$ has the following expansion (see \cite{LP-87} and \cite[Proposition 2.1]{PX}):
	\[
		\begin{aligned}
		&\left((\exp_P^g )^\ast g\right) (x) \left[ F_{P,\a}, F_{P,\b} \right] 
		\\
		 = \, & \delta_{\a \b} + \frac{1}{3} g_P \left[ R_P  (\Xi , F_{P,\a} ) \Xi , F_{P,\b}  \right]
		+ \frac{1}{6} g_P\left[ \nabla_{\Xi} R_P  (\Xi , F_{P,\a}) \Xi, F_{P,\b}  \right] 
		 + \frac{1}{20} g_P \left[ \nabla_{\Xi} \nabla_{\Xi} R_P 
		(\Xi , F_{P,\a}) \Xi, F_{P,\b} \right] 
		\\
		& 
		+ \frac{2}{45} g_P \left[ R_P (\Xi ,F_{P,\a}) \Xi , F_{P,\g} \right] 
		g_P \left[ R_P ( \Xi , F_{P,\b} ) \Xi , F_{P,\g} \right] 
		+ {\rm Rem}_{\a \b}(x,P)
		\end{aligned}
	\]
where $\Xi := x^\alpha F_{P,\a}$ and $| D_P^k {\rm Rem}_{\a \b} (x,P) | \leq C |x|^5$ for every 
$|x| \leq \rho_0$, $P \in \ov{V}_0$ and $k \in \N$. 
Since the expansions \eqref{eq:HSeP} and \eqref{eq:dvolSigmaePH} in the $C^0$ sense are obtained in 
\cite[Lemmas 3.3 and 3.5]{M-JGA}, by the smooth dependence on $P$ of the metric, we can also show 
\eqref{eq:HSeP} and \eqref{eq:dvolSigmaePH} in the $C^2$ sense.

	For \eqref{eq:WE}, it follows from \eqref{eq:HSeP} that at $q = q^\a F_{P,\a}$ with $q \in S^2$,
	\[
		H^2_{\Sigma_{\e,P}^{0}} (q)
		= \frac{1}{\e^2} 
		\left( 4 - \frac{4}{3} \e^2 \Ric_P( q, q ) + \cO \e^3 + O_{C^2}(\e^4)  \right).
	\]
Noting $D_P^k \cO (-q) = - D_P^k \cO(q) $ and $\int_{S^2} \cO d \sigma_{S^2} = 0$, 
we observe from \eqref{eq:dvolSigmaePH} that 
	\[
		\begin{aligned}
			W_g \left( \Sigma_{\e,P}^{0} \right) 
			&= \int_{S^2} \left[ 4 - 2\e^2 \Ric_P (q,q) + \cO \e^3 + O_{C^2}(\e^4) \right] 
			d \sigma_{S^2}= 16\pi - \frac{8\pi }{3} \e^2 \Sc_P + O_{C^2} (\e^4).
		\end{aligned}
	\]

	Finally, for \eqref{eq:WEdiff}, we notice that 
	\[
		\begin{aligned}
			& \frac{\partial}{\partial \e} \left((\exp_P^g )^\ast g\right) (\e x) \left[ F_{P,\a}, F_{P,\b} \right] 
			\\
		= \, &  \frac{2\e}{3} g_P \left[ R_P  (\Xi , F_{P,\a} ) \Xi , F_{P,\b}  \right]  
			+ \frac{\e^2}{2} g_P\left[ \nabla_{\Xi} R_P  (\Xi , F_{P,\a}) \Xi, F_{P,\b}  \right] 
			+ \frac{\e^3}{5} g_P \left[ \nabla_{\Xi} \nabla_{\Xi} R_P 
			(\Xi , F_{P,\a}) \Xi, F_{P,\b} \right] 
			\\
			& 
			+ \frac{8\e^3}{45} g_P \left[ R_P (\Xi ,F_{P,\a}) \Xi , F_{P,\g} \right] 
			g_P \left[ R_P ( \Xi , F_{P,\b} ) \Xi , F_{P,\g} \right] 
			+ \frac{\partial }{\partial \e} {\rm Rem}_{\a \b}(\e x,P).
		\end{aligned}
	\]
We also remark that the last term satisfies 
	\[
		\left| D_x^k D_P^\ell \frac{\partial }{\partial \e} {\rm Rem}_{\a \b}(\e x,P)  \right| 
		\leq C_{k,\ell} \e^4
	\]
for any $k,\ell \in \N$. From these facts and the proof for \eqref{eq:WE}, it is not difficult to check \eqref{eq:WEdiff} 
and we complete the proof. 
\end{proof}

\begin{proof}[Proof of Theorem \ref{thm:foliation}]
>From Theorem \ref{101} we know that, for $\e_0>0$ small enough, for each $\e \in (0,\e_0)$
there exists an area-constrained Willmore  sphere
\begin{equation}\label{eq:defSigmae}
\Sigma_{\e}:=\Sigma_{\e,P_{\e}}[\varphi_{\e,P_{\e}}] := \exp_{P_{\e}}^{g} ( \e S^2[\varphi_{\e,P_{\e}}] ), \quad |\Sigma_\e|_g = 4 \pi \e^2,
\quad P_\e \to P_0. 
\end{equation}
Recall also that $\Sigma_\e$ is a critical point of $\Phi_\e$. 
We also  denote $\Sigma_{\e,P}^{0}:=\Sigma_{\e,P}[0]$.
\\

\textbf{Step 1}. For a suitable neighborhood $U$ of $P_0$ and $j=0,1$, 
	\begin{equation}\label{eq:nablaPhinablaScC1}
		\left\| D_\e^j \left( D_P \Phi_{\e} + \frac{8\pi}{3} \e^{2} \, D_P \Sc_{(\cdot)} \right)   \right\|_{C^{1}(U)}\leq C \e^{4-j}.
	\end{equation}

\noindent
To this aim, for all $P\in U \subset \ov{V}_0$, we first remark that 
\begin{equation*}\label{eq:defFNablaSc}
	\begin{aligned}
		& \left| D_\e^j \left( D_P \Phi_{\e}(P) + \frac{8\pi}{3} \e^{2} \, D_P \Sc_P \right)   \right|
		\\
		\leq & \,  
		\left| D_\e^j \left[D_P \Phi_{\e}(P) - D_{P}(W(\Sigma_{\e,P}^{0}))  \right] \right| 
		+ \left| D_\e^j \left[ D_{P}(W(\Sigma_{\e,P}^{0})) + \frac{8\pi}{3} \e^{2} \, D_P \Sc_P  \right]   \right|.
	\end{aligned}
\end{equation*}
Recalling $W_g(\Sigma^0_{\e,P}) = W_g ( \exp^g_P(\e S^2) ) = W_{g_\e} (\exp^{g_\e}_P (S^2))$ and \eqref{eq:diff-e0}, 
we know that 
	\[
		D_P \left( \Phi_\e - W_g \left( \Sigma^0_{\e,P} \right) \right) 
		= \int_0^{1} D_P \left( W_{g_\e}' \left( S_{\e,P}[ s \varphi_{\e,P}] \right) \left[ \psi_{\e,P,s} \right]  \right) d s.
	\]
By Lemmas \ref{301} and \ref{302}, Proposition \ref{303} and \eqref{eq:psi}, we may observe that for $j=0,1$, 
\begin{equation*}\label{eq:nablaPhiDpW}
 \left\| D_\e^j D_{P} \left( \Phi_{\e} - W_g\left(\Sigma_{\e,(\cdot)}^{0}\right)\right) \right\|_{C^{1}(U)}\leq C\e^{4-j}. 
\end{equation*}
Thus, in order to get \eqref{eq:nablaPhinablaScC1} it is enough to prove that
\begin{equation*}\label{eq:nablaPhiDpW-2}
 \left\| D_\e^j D_{P} \left(W_g \left(\Sigma_{\e,(\cdot)}^{0}\right) + \frac{8\pi}{3} \e^{2} \, \Sc_{(\cdot)} \right) \right\|_{C^{1}(U)}\leq C\e^{4-j}
\end{equation*}
for $j=0,1$. 
But this is easily seen from $D_P 16\pi = 0$ and Lemma \ref{lmm:geomExp}. 
Thus, Step 1 holds.


\medskip

\textbf{Step 2}. ${\rm d}_{g}(P_{\e},P_{0})\leq C \e^{2}$. 

\noindent
Since $P_\e$ is a critical point of $\Phi_\e$ and we may assume $P_\e \in U$, 
Step 1 gives 
\begin{align*}
0=|D_P \Phi_{\e} (P_{\e})|\geq |D_P \Sc_{P_{\e}}|\e^{2}- C\e^{4}\geq  C_{\Sc} \, {\rm d}_{g}(P_{0},P_{\e}) \e^{2} - C\e^{4},
\end{align*}
which yields the claim  ${\rm d}_{g}(P_{\e},P_{0})\leq (C/C_{Sc}) \e^{2}$.

\medskip

\textbf{Step 3}.  The surfaces $\{\Sigma_{\e}\}_{\e\in (0,\e_{0})}$, defined  in \eqref{eq:defSigmae}, form a foliation of $U\setminus P_{0}$. 
\\We will work with the following parametrisation of $\Sigma_{\e}$:
\[
		S^2[\varphi_{\e,P_{\e}}] := 
		\left\{ \big(1+ \varphi_{\e,P_{\e}}(q)\big) \, q \ |\ q \in S^2  \right\} \subset \R^3\simeq T_{P_{\e}}M, \quad 
		\Sigma_{\e,P_{\e}}[\varphi_{\e,P_{\e}}] := 
		\exp_{P_{\e}}^{g} ( \e S^2[\varphi_{\e,P_{\e}}] ).
	\]
	 Set
\begin{equation*}\label{eq:defF}
F:S^2\times (0,\varepsilon_0) \to M, \quad 
F(q,\e):= \exp_{P_{\e}}^{g}\Big( \e\big(1+ \varphi_{\e,P_{\e}}(q)\big) q   \Big).
\end{equation*}
For $\e_{0}>0$ small enough, we claim that $F(S^2\times (0,\varepsilon_0))=U\setminus\{P_{0}\}$ and that $F$ is a diffeomorphism onto its image. 
\\First of all we show that $F$ is smooth. Thanks to Proposition \ref{303}, the map $(\e, P)\mapsto \varphi_{\e,P}$ is smooth. 
Also the map $\e\mapsto P_{\e}$  is smooth. Indeed $P_{\e}$ is defined as the unique solution in $U$ of 
	\[
		0= D_P \Phi_{\e}=  -  \frac{8\pi}{3} \e^{2}  D_P\Sc +O_{C^{1}}(\e^{4})
	\]
where in the last identity we used \eqref{eq:nablaPhinablaScC1}; since by assumption $P_{0}$ is a non-degenerate critical point of $\Sc$, the Implicit Function Theorem guarantees the smoothness of $\e\mapsto P_{\e}$.
We thus conclude that $F$ is smooth, as composition of smooth maps.
Moreover, differentiating $0= D_P \Phi_\e(P_\e)$ in $\e$ and using \eqref{eq:nablaPhinablaScC1}, we obtain 
	\[
		D_P^2 \Phi_\e (P_\e) \frac{d P_\e}{d \e} = - D_\e D_P \Phi_\e (P_\e) = \frac{16\pi}{3} \pi \e D_P \Sc_{P_\e} 
		+ O(\e^3). 
	\]
By $D_P \Sc |_{P=P_0} = 0$, ${\rm d}_g (P_\e,P_0) = O(\e^2)$ due to Step 2 and 
$(D_P^2 \Phi_\e )^{-1}|_{P=P_\e} = O(\e^{-2})$ thanks to \eqref{eq:nablaPhinablaScC1}, we observe that 
\begin{equation}\label{eq:dPede}
\left|\frac{dP_{\e}}{d\e}\right|\leq C \e \quad  \text{for } \e\in (0,\e_{0}).
\end{equation}

We next claim that there exists $C>0$ (independent of $\e$) such that
\begin{equation}\label{eq:claimdFde}
\left|g\left( \frac{\partial}{\partial \e} F(q,\e), n_{\e}(q) \right)-1\right| \leq C\e \quad \text{for every  } (q,\e) \in S^2\times (0,\varepsilon_0),
\end{equation}
where $n_{\e}(q)$ is the outer unit normal to $\Sigma_{\e}=F(S^{2}\times \{\e\})$ at $F(q,\e)$.
To this aim we compute:
\begin{align}
 \frac{\partial F}{\partial \e}(q,\e)&= 
 \left( D_P \exp_P^g (v) \right) \Big|_{(P,v) = (P_\e , \e (1+\varphi_{\e,P_\e} ) q) } \frac{dP_{\e}}{d\e} 
 \nonumber
 \\
 & \quad 
 +  (D_{v} \exp_{P_{\e}}^{g} (v))\big|_{v= \e(1+ \varphi_{\e,P_{\e}}(q)) q} \left( \Big( 1+\varphi_{\e,P} + \frac{\partial \varphi_{\e,P}}{\partial \e}  + D_{P} \varphi_{\e,P} \frac{d P_{\e}}{d\e}\Big) q   \right)\Big|_{P=P_{\e}} \nonumber \\
 &= O(\e)+ (D_{v} \exp_{P_{\e}}^{g})\big|_{v= \e(1+ \varphi_{\e,P_{\e}}(q)) q} \big(q+O(\e) \big), \label{eq:dFde}
\end{align}
where in the second line we used  Proposition \ref{303} and \eqref{eq:dPede}. 
\\The claim \eqref{eq:claimdFde} follows by combining \eqref{eq:estgHn} with \eqref{eq:dFde}.
\\Since from Step 2 we know that ${\rm d}(P_{\e}, P_{0})\leq C\e^{2}$,  the estimate in \eqref{eq:claimdFde} ensures that $F(S^2\times (0,\varepsilon_0))=U\setminus\{P_{0}\}$ and that $F$ is a diffeomorphism onto its image; in other words, $F $ induces a foliation of $U\setminus\{P_{0}\}$  by the area-constrained Willmore spheres $\Sigma_{\e}=F(S^{2}\times\{\e\})$.
 \\
  
 \textbf{Step 3}.  The foliation is regular at $\e=0$.
 \\Fix a system of normal coordinates of $U$ centered at $P_{0}$, indentify $U$ with  an open  subset of $\R^3$ and call $F_\varepsilon:=\frac{1}{\varepsilon} F(\cdot, \varepsilon):S^2 \to \R^3$.  
 Since ${\rm d}_{g}(P_{\e},P_{0})\leq C \e^{2}$ by Step 2 and $\|\varphi_{\e}\|_{C^{5,\alpha}(S^{2})}\leq C\e^{2}$ by \eqref{eq:propSe}, it follows that  the immersions  $F_\varepsilon$ converge in $C^{5,\alpha}$-norm to the round unit sphere of $\R^3$, as $\varepsilon \searrow 0$. The convergence in $C^{k}(S^{2})$-norm, for every $k\in \N$, follows from a standard bootstrap argument thanks to the ellipticity of $W'(S^{2})$. 
 \\
  
 \textbf{Step 4}.  The foliation is unique among Willmore spheres of energy $<32 \pi$.
\\ Let  $V \subset U$ be  another neighborhood of $P \in M$  such that $V\setminus \{P\}$  is foliated by area-constrained Willmore spheres $\Sigma'_\varepsilon$ having area  $4\pi \varepsilon^2$, $\varepsilon\in (0,\varepsilon_1)$, and satisfying $\sup_{\varepsilon\in (0,\varepsilon_1)} W_{g}(\Sigma'_\varepsilon) < 32 \pi$. 
\\By \cite{LauMon}, for $\e$ small enough, $\Sigma'_\varepsilon$ are normal graphs over geodesic spheres, i.e. there exist $P_{\e}'\in V, \varphi_{\e}'\in C^{5,\alpha}(S^{2})$ such that
$$
\Sigma'_\varepsilon=\Sigma_{\e,P_{\e}'}[\varphi_{\e}'] := \exp_{P_{\e}'}^{g} ( \e S^2[\varphi_{\e}'] ).
$$ 
Since by assumption $\Sigma'_\varepsilon$ is an area-constrained Willmore sphere, by the uniqueness statement in the Implicit Function Theorem, Proposition \ref{303} implies that $\varphi_{\e}'=\varphi_{\e, P_{\e}'}$.
Using again that $\Sigma'_\varepsilon$ is an area-constrained Willmore sphere, we infer that $P_{\e}'$ is a critical point of the reduced functional $V\ni P\mapsto \Phi_{\e}(P):=W_{g}(\Sigma_{\e,P}[\varphi_{\e, P}])$.
But since by assumption $P_{0}$ is a non-degenerate critical point of $\Sc$, the arguments above \eqref{eq:dPede} yield that $\Phi_{\e}$ has a unique critical point in $V$. We conclude that $P_{\e}'=P_{\e}$ and thus 
 $\Sigma_\varepsilon=\Sigma'_\varepsilon$ for  $\varepsilon$ small enough.
%
%

\end{proof}


\end{document}